\documentclass[10pt]{amsart}

\usepackage[dvipsnames]{xcolor}

\usepackage[colorlinks=true]{hyperref}
\usepackage{amssymb,amsthm}
\usepackage[msc-links]{amsrefs}

\usepackage{setspace}

\usepackage{tikz}
\usetikzlibrary{cd}

\usepackage[capitalise]{cleveref} 

\title[Topological Models of Abstract Commensurators]{Topological Models of Abstract Commensurators}
\author{Edgar A. Bering IV}
\address{Department of Mathematics, Technion---Israel Institute of Technology, Haifa, Israel}
\email{bering@campus.technion.ac.il}
\author{Daniel Studenmund}
\address{Department of Mathematical Sciences, Binghamton University, Binghamton, New York}
\email{daniel@math.binghamton.edu}
\thanks{EB was supported by the Azrieli Foundation}
\thanks{DS was supported in part by NSF grant DMS-1547292}

\newtheorem{theorem}{Theorem}
\newtheorem*{introtheorem}{Theorem}
\newtheorem{corollary}{Corollary}
\newtheorem{proposition}{Proposition}
\newtheorem{lemma}{Lemma}

\theoremstyle{definition}
\newtheorem{definition}{Definition}

\newtheorem*{convention}{Convention}
\newtheorem{fact}{Fact}

\newcommand{\oldqedsymbol}{\qedsymbol}

\theoremstyle{remark}
\newtheorem{remark}{Remark}

\newcommand{\LabDef}[1]{\label{defn:#1}}
\newcommand{\RefDef}[1]{Definition~\ref{defn:#1}}

\newcommand{\LabThm}[1]{\label{thm:#1}}
\newcommand{\RefThm}[1]{Theorem~\ref{thm:#1}}

\newcommand{\LabProp}[1]{\label{prop:#1}}
\newcommand{\RefProp}[1]{Proposition~\ref{prop:#1}}

\newcommand{\LabLem}[1]{\label{lem:#1}}
\newcommand{\RefLem}[1]{Lemma~\ref{lem:#1}}

\newcommand{\LabCor}[1]{\label{cor:#1}}
\newcommand{\RefCor}[1]{Corollary~\ref{cor:#1}}

\newcommand{\LabEq}[1]{\label{eq:#1}}
\newcommand{\RefEq}[1]{(\ref{eq:#1})}

\newcommand{\LabSec}[1]{\label{sec:#1}}
\newcommand{\RefSec}[1]{Section~\ref{sec:#1}}

\DeclareMathOperator{\Comm}{Comm}
\DeclareMathOperator{\Aut}{Aut}

\DeclareMathOperator{\Mod}{Mod}
\DeclareMathOperator{\HMod}{\mathcal{E}}
\DeclareMathOperator{\Homeo}{Homeo}
\DeclareMathOperator{\id}{id}
\DeclareMathOperator{\QI}{QI}

\DeclareMathOperator{\HE}{HE}

\newcommand{\cent}[2]{C_{#2}(#1)}
\providecommand{\conj}[1]{\ensuremath{c_{#1}}}

\DeclareMathOperator{\blip}{\widetilde{Lip}}
\DeclareMathOperator{\injrad}{inj.rad}

\newcommand{\fidx}{\leq_f}
\newcommand{\nfidx}{\trianglelefteq_f}

\newcommand{\from}{\colon}

\newcommand{\homeo}{\approx}
\newcommand{\bdry}{\partial}
\newcommand{\hypspace}{\mathbb{H}}

\newcommand{\ucov}[1]{\widetilde{#1}}
\newcommand{\sol}[1]{\widehat{#1}}
\newcommand{\blmap}{\ell}

\newcommand{\profin}[1]{\widehat{#1}}
\DeclareMathOperator{\pfi}{\iota}

\newcommand{\solmod}[2]{\profin{#1}\times_{#1} \ucov{#2}}

\newcommand{\commincl}{\hat{\iota}}
\newcommand{\Gn}{G_{\leq n}}

\newcommand{\ficol}[1]{\operatorname{\mathcal{F}}(#1)}
\newcommand{\bfG}{\mathbf{G}}

\newcommand{\adhocf}{\hat{f}}

\newcommand{\ucobp}{p}

\newcommand{\minelt}{m}

\DeclareMathOperator{\dpro}{\profin{d}}
\DeclareMathOperator{\dsol}{\sigma}

\providecommand{\abs}[1]{\ensuremath{\left\lvert#1\right\rvert}}

\providecommand{\set}[1]{\ensuremath{\left\lbrace#1\right\rbrace}}
\providecommand{\floor}[1]{\ensuremath{\left\lfloor#1\right\rfloor}}

\newcommand\suchthat{ \mathrel{}\middle| \mathrel{}}

\newcommand\restr[2]{{
  \left.\kern-\nulldelimiterspace 
  #1 
  \vphantom{\big|} 
  \right|_{#2} 
  }}

\newcommand{\catname}[1]{\mathsf{#1}}

\newcommand{\Grp}{\catname{Grp}}
\newcommand{\Cpt}{\catname{Cpt}}

\newcommand{\Top}{\catname{Top}}

\newcommand{\HTop}{\catname{HTop}}
\newcommand{\HCpt}{\catname{HCpt}}
\newcommand{\pro}[1]{\text{\normalfont{pro-}}#1} 
\let\oldmarginpar\marginpar
\renewcommand{\marginpar}[1]{
\oldmarginpar{\footnotesize\color{red} #1}}

\begin{document}

\begin{abstract}
The full solenoid over a topological space $X$ is the inverse limit of all finite covers. When $X$ is a compact Hausdorff space admitting a locally path connected universal cover, we relate the pointed homotopy equivalences of the full solenoid to the abstract commensurator of the fundamental group $\pi_1(X)$. The relationship is an isomorphism when $X$ is an aspherical CW complex. If $X$ is additionally a geodesic metric space and $\pi_1(X)$ is residually finite, we show that this topological model is compatible with the realization of the abstract commensurator as a subgroup of the quasi-isometry group of $\pi_1(X)$. This is a general topological analogue of work of Biswas, Nag, Odden, Sullivan, and others on the universal hyperbolic solenoid, the full solenoid over a closed surface of genus at least two.
\end{abstract}

\vspace*{-1.4cm}

\maketitle
    
\section{Introduction}

The abstract commensurator of a group $G$, denoted $\Comm(G)$, is the collection of isomorphisms $\phi\from H\to K$ between finite-index subgroups $H, K \leq G$, modulo equivalence which identifies isomorphisms agreeing on a finite-index domain. When $G$ is infinite, $\Comm(G)$ is a natural relaxation of the automorphism group $\Aut(G)$, but is often more difficult to study. For example, it is not known whether $\Comm(F_2)$ is simple, where $F_2$ is a two-generated free group. Our goal is to provide topological and metric perspectives on $\Comm(G)$ parallel to those used to study $\Aut(G)$.

Given a group $G$ with Eilenberg-MacLane space $X = K(G,1)$, the automorphism group $\Aut(G)$ is topologically modelled by $\HMod(X,\ast)$, the group of pointed homotopy classes of pointed homotopy equivalences of $X$. When $G$ is finitely generated, $\Aut(G)$ is metrically modelled by a natural map to $\QI(G)$, the quasi-isometry group of $G$. If $X$ is a compact geodesic metric space, these two models of $\Aut(G)$ are compatible: each pointed homotopy equivalence of $X$ lifts to a proper homotopy equivalence of the universal cover $\ucov{X}$ preserving the $G$-orbit of the basepoint, which determines an element of $\QI(G)$ by the Milnor-Schwarz lemma.

The metric model of $\Aut(G)$ extends to a natural map $\Comm(G) \to \QI(G)$, a metric model of $\Comm(G)$. This article develops a topological model of $\Comm(G)$ compatible with this metric model and analogous to the homotopy model of $\Aut(G)$.

\subsection{Statement of results} Our results apply to fundamental groups of compact Hausdorff spaces admitting a locally path-connected universal cover; spaces satisfying the latter condition are called {\em unloopable}. Following classic constructions of McCord \cite{mccord}, for any such space $(X,\ast)$ we construct the \emph{full solenoid} over $X$ as the inverse limit $(\sol{X},\ast)$ of all finite-sheeted covers of $(X,\ast)$. 

\begin{introtheorem}[\ref{thm:ptd-homotopy-comm}]
    Suppose $(X,\ast)$ is an unloopable pointed compact Hausdorff space. Then there is a homomorphism
    \[ C\from \HMod(\sol{X},\ast) \to \Comm(\pi_1(X,\ast)) \]
    If $X$ is homotopy equivalent to an aspherical CW complex, this map is an isomorphism.
\end{introtheorem}

For example, let $X = T^n$ be an $n$-dimensional torus with $G = \pi_1(X,\ast) \cong \mathbb{Z}^n$. It is an exercise to prove that $\Comm(\mathbb{Z}^n) \cong \operatorname{GL}_n(\mathbb{Q})$, and so there is an isomorphism $\HMod \left( \sol{T^n}, \ast \right) \cong \operatorname{GL}_n(\mathbb{Q})$.

After discussing abstract commensurators in \RefSec{abcomm} and reviewing the structure of full solenoids in \RefSec{full-solenoids}, we prove \RefThm{ptd-homotopy-comm} in \cref{sec:shape-of-solenoids}.
The proof combines the structure of the full solenoid as a fiber bundle over $X$ with totally disconnected fibers with standard facts in shape theory applied to the inverse system of covering spaces used to define $\sol{X}$. We present the arguments with care, starting with a review of inverse systems in \RefSec{inverse-systems}, because the techniques may be unfamiliar to some readers. 

Continuing the analogy with the automorphism group, for a residually finite group $G$ we connect the topological model of $\Comm(G)$ to its realization in $\QI(G)$ in \RefSec{he-qi}.
To do so, we assume the base space $X$ is a geodesic metric space and then use the metric structure on $\sol{X}$ as a space foliated by leaves quasi-isometric to $G$, which we develop in \RefSec{metric-notions}.

\begin{introtheorem}[\ref{thm:qi-map-factors}]
Suppose $(X,\ast)$ is an unloopable compact geodesic metric space with residually finite fundamental group $G = \pi_1(X,\ast)$. There is a homomorphism $Q\from \HMod(\sol{X},\ast)\to \QI(G)$ which factors through the homomorphism $C\from \HMod(\sol{X},\ast)\to \Comm(G)$ of \RefThm{ptd-homotopy-comm}.
\end{introtheorem}

We were motivated to study full solenoids by work of Sullivan, with Biswas and Nag, who initiated the study of the universal hyperbolic solenoid $\sol{\Sigma}$~\cites{biswas-nag-sullivan, sullivan-universal, sullivan14}, the full solenoid over a closed surface $\Sigma$ of genus $g\geq 2$. Odden then proved that the pointed mapping class group $\Homeo(\sol{\Sigma},\ast)/\Homeo_0(\sol{\Sigma},\ast) \cong \Comm(\pi_1(\Sigma))$ provides a topological model of the abstract commensurator of the fundamental group of a closed surface~\cite{odden}.
Odden connects the homeomorphisms of the solenoid to the action of $\Comm(\pi_1(\Sigma))$ on the boundary of the hyperbolic plane.
\RefSec{boundarycomm} records a generalization in the setting of hyperbolic groups, relating the action of $\Comm(G)$ on $\partial G$ to the solenoid model.

\begin{introtheorem}[\RefCor{sol-hyp-bdry}]
 Suppose $G\cong \pi_1(X,\ast)$, where $(X,\ast)$ is an unloopable compact geodesic metric space homotopy equivalent to an aspherical CW complex. Suppose $G$ is a residually finite, torsion-free, non-elementary hyperbolic group. Then the homomorphism $Q\from \HMod(\sol{X},\ast)\to QI(G)$ of \RefThm{qi-map-factors} induces isomorphisms
 \[
 \HMod(\sol{X},\ast) \cong \Comm(G) \cong \Comm_{\Homeo(\bdry G)}(G).
 \]
\end{introtheorem}

\subsection{Relationship with known commensurator computations} \RefCor{sol-hyp-bdry} fits into a familiar scheme of understanding $\Comm(G)$ by finding a map $\Comm(G) \to \Aut(X)$ for some object $X$ on which $G$ acts faithfully, then proving the map is an isomorphism onto the relative commensurator $\Comm_{\Aut(X)}(G)$.  When $G$ is a sufficiently nice branch group acting on a tree $T$, such as Grigorchuk's group of intermediate growth, R\"over \cite{rover} gave an isomorphism $\Comm(G) \cong \Comm_{\Homeo(\bdry T)}(G)$. When $G = F$ is Thompson's group, Burillo, Cleary, and R\"over \cite{burillo-cleary-rover} identified an isomorphism $\Comm(F) \cong \Comm_{\operatorname{PL}_2^\pm(\mathbb{R})}(F)$, where $\operatorname{PL}_2^\pm(\mathbb{R})$ is a certain group of piecewise linear homeomorphisms of $\mathbb{R}$. When $G = \Mod(\Sigma)$ is the extended mapping class group of a closed surface of genus $g \geq 3$ and $C(\Sigma)$ is the curve complex of $\Sigma$, Ivanov \cite{ivanov} gave an isomorphism $\Comm( \Mod(\Sigma) ) \cong \Comm_{\Aut(C(\Sigma))}( \Mod(\Sigma) ) $. This has been generalized to other subgroups of mapping class groups, notably in recent work of Brendle and Margalit \cite{brendle-margalit}.

In each of these cases, structure of $\Aut(X)$ has allowed for more information about $\Comm(G)$ to be determined. Obtaining the strongest possible conclusion, Ivanov proved a rigidity statement $\Aut(C(\Sigma)) \cong \Mod(\Sigma)$, which has the consequence that the natural map $\Mod(\Sigma)\to \Comm( \Mod(\Sigma) )$ is an isomorphism.

\subsection{Bibliographic remarks}\label{sec:bibliographic-remarks}

Solenoids are well-studied objects in classical topology, as a rich source of examples and counterexamples, as well as in dynamics, where they arise naturally in the study of foliations and actions on Cantor sets. Their appearance in dynamics goes back to work of Smale \cite{smale} and Williams \cite{williams}, who studied solenoids arising as the inverse limit of iterated self-immersions of branched manifolds. 

McCord's work, summarized in \RefThm{metric-model-new} in \RefSec{full-solenoids}, identifies $\sol{X}$ with the suspension of the action of $\pi_1(X,\ast)$ on its profinite completion by right translation. This is an example of a minimal, equicontinuous group action on a Cantor set. From this point of view, the full solenoid is an example of a foliated space called a matchbox manifold. Such dynamical systems have been studied systematically by Clark, Dyer, Hurder, and Lukina~\cites{clark-hurder-lukina-bing, clark-hurder-lukina-classification, dyer-hurder-lukina-cantor}.

The structure of $\sol{X}$ as a principal bundle over $X$, which is used in the proof of \RefThm{ptd-homotopy-comm}, has been further studied by Clark and Fokkink \cite{clark-fokkink-bihomogeneity} in the case that $X$ is a closed manifold. Weak solenoids over closed manifolds were shown by Fokkink and Oversteegen to be fiber bundles with profinite group fiber \cite{fokkink-oversteegen-homogeneous}*{\S7}.

\subsection*{Acknowledgements}

We thank Matt Brin, Benson Farb, Ross Geoghegan, and Olga Lukina for enlightening discussions regarding commensurators and solenoids. We are grateful to James Belk and Bradley Forrest for describing the inverse systems approach to solenoids and commensurators appearing in Sections \ref{sec:abcomm} and \ref{sec:shape-of-solenoids}. Steven Hurder provided many helpful conversations, comments on an early draft of this article, and great hospitality.

\section{Inverse systems and pro-categories}
\LabSec{inverse-systems}

Solenoids are by definition inverse limits of certain topological spaces. In order to relate their topology to the abstract commensurator we start by recalling the relevant categorical framework. 
We follow the language and conventions of Marde\v{s}i\'{c} and Segal \cite{mardesic-segal}.

Given a category $\mathcal{C}$ an \emph{inverse system} in $\mathcal{C}$ consists of a directed set $\Lambda$, a collection of objects $\{ X_\lambda\}_{\lambda\in\Lambda}$ and bonding morphisms $p_{\lambda\lambda'}\from X_{\lambda'}\to X_\lambda$ for each $\lambda \leq \lambda'$ satisfying $p_{\lambda\lambda} = \id_{X_\lambda}$ and $p_{\lambda\lambda'}\circ p_{\lambda'\lambda''} = p_{\lambda\lambda''}$ for all $\lambda \leq \lambda' \leq \lambda''$.
To ease the burden of notation, when referring to the entire system we will use boldface without indexing information when such information is not needed.

Given two inverse systems $(\{X_\lambda\},p_{\lambda\lambda'},\Lambda)$  and $(\{Y_\mu\}, q_{\mu\mu'}, M)$
a \emph{morphism} $\mathbf{f} = (f_\mu, \phi) \from \mathbf{X}\to\mathbf{Y}$ is a function $\phi\from M\to \Lambda$ and a collection of
morphisms $\{ f_\mu \from X_{\phi(\mu)}\to Y_\mu\}_{\mu\in M}$ that are compatible with the bonding morphisms. That is, for all $\mu \leq \mu' \in M$ there exists $\lambda \geq \phi(\mu),\phi(\mu')$ such
that 
\[f_\mu\circ p_{\phi(\mu)\lambda} = q_{\mu\mu'}\circ f_{\mu'} \circ p_{\phi(\mu')\lambda'}.\] Two
morphisms $(f_\mu,\phi), (g_\mu,\psi) \from \mathbf{X}\to\mathbf{Y}$ are \emph{equivalent} if for all
$\mu$ there exists $\lambda\geq \phi(\mu),\psi(\mu)$ such that
\[ f_\mu\circ p_{\phi(\mu)\lambda} = g_\mu\circ p_{\psi(\mu)\lambda}. \]

\begin{definition}\LabDef{procategory}
    A \emph{pro-category} over a category $\mathcal{C}$, denoted $\pro{\mathcal{C}}$, has as objects inverse systems
    in $\mathcal{C}$, and as morphisms equivalence classes of system morphisms. 
\end{definition}

Observe that there is a functor from $\mathcal{C}$ to pro-$\mathcal{C}$ that sends each object $X$ to the rudimentary system $(X)$ indexed by a singleton and that this allows us to treat $\mathcal{C}$ as a full subcategory of $\pro{\mathcal{C}}$. In the interests of easing notation we will typically omit the system notation from rudimentary systems. 

If $(\Lambda, \leq)$ is a directed set and $M \subseteq \Lambda$, say that $M$ is {\em cofinal} if, for every $\lambda \in \Lambda$, there is some $\mu \in M$ such that $\lambda\leq \mu$. Given any inverse system $\mathbf{X} = ( \{X_\lambda\}, p_{\lambda \lambda'}, \Lambda )$ in $\pro{\mathcal{C}}$ and directed subset $M \subseteq \Lambda$, then $\mathbf{X}^{M} = ( \{X_\mu\}, p_{\mu\mu'}, M )$ is an inverse system and there is a restriction map of inverse systems $\mathbf{i} \from \mathbf{X} \to \mathbf{X}^M$. 
\begin{lemma}[\cite{mardesic-segal}*{Ch I, \S1.1, Theorem 1}]
\LabLem{lem:cofinal-system}
If $\mathbf{X} = ( \{X_\lambda\}, p_{\lambda \lambda'}, \Lambda )$ is an inverse system in $\mathcal{C}$ and $M \subseteq \Lambda$ is cofinal, then the restriction $\mathbf{i}\from \mathbf{X} \to \mathbf{X}^M$ is an isomorphism in $\pro{\mathcal{C}}$.
\end{lemma}

A directed set $(\Lambda, \leq)$ is {\em cofinite} if, for each $\lambda \in \Lambda$, the set $\set{ \lambda' \in \Lambda \suchthat \lambda' \leq \lambda}$ is finite. All
directed sets we consider are cofinite.
\begin{lemma}[\cite{mardesic-segal}*{Ch I, \S1.2, Lemma 2}]
\LabLem{lem:cofinite-morphisms}
Suppose $\mathbf{X} = (\{X_\lambda\},p_{\lambda\lambda'},\Lambda)$  and $\mathbf{Y} = (\{Y_\mu\}, q_{\mu\mu'}, M)$ are inverse systems in $\mathcal{C}$ and $\mathbf{f} = (f_\mu, \phi) \from \mathbf{X} \to \mathbf{Y}$ is a morphism of inverse systems. If $M$ is cofinite, then $\mathbf{f}$ is equivalent to a morphism of systems $(g_\mu, \psi) \from \mathbf{X} \to \mathbf{Y}$ such that $\psi \from M \to \Lambda$ increases and for any $\mu \leq \mu'$ the following diagram commutes:
    \begin{center}
    \begin{tikzcd}
    X_{\psi(\mu')}\ar{r}{g_{\mu'}}\ar[swap]{d}{p_{\psi(\mu)\psi(\mu')}} & Y_{\mu'} \ar{d}{q_{\mu\mu'}} \\
    X_{\psi(\mu)}\ar{r}{g_\mu} & Y\mu
    \end{tikzcd}
    \end{center}
\end{lemma}

We use the following notation for categories. $\Top$ is the category of topological spaces and continuous maps, and $\Cpt$ is the category of compact Hausdorff spaces and continuous maps.  $\HTop$ is the category whose objects are topological spaces and whose morphisms are homotopy classes of continuous maps. $\HCpt$ is defined analogously.  We use $\Top_\ast$, $\Cpt_\ast$, $\HTop_\ast$, and $\HCpt_\ast$ to denote the respective pointed categories, in which each space is equipped with a basepoint, and all continuous maps and homotopies preserve basepoints. $\Grp$ is the category of groups with homomorphisms. The notation $\mathcal{C}(A, B)$ is used for the set of $\mathcal{C}$ morphisms from $A$ to $B$.

\section{Abstract commensurators} \LabSec{abcomm}

Write $H \fidx G$ to mean that $H$ is a finite-index subgroup of a group $G$, and $H \nfidx G$ to mean moreover that $H$ is normal. 
A \emph{partial automorphism} of a group $G$ is an isomorphism between two finite index subgroups of $G$. If
$\phi_1 \from H_1 \to K_1$ and $\phi_2\from H_2 \to K_2$ are partial
automorphisms declare $\phi_1$ \emph{equivalent} to $\phi_2$, denoted $\sim$,
if there exists $H_3 \fidx H_1\cap H_2$ such that $\phi_1$ and $\phi_2$ agree
on $H_3$.

The \emph{abstract commensurator} of a group $G$ is the group 
\[ \Comm(G) = \set{ \phi\from H \stackrel{\sim}{\to} K \suchthat H, K \fidx G}/\sim \]
of equivalence classes of partial automorphisms. 
$\Comm(G)$ is a group under the rule $[\phi] \circ [\psi] = [\phi' \circ \psi']$ for any $\phi \sim \phi'$ and $\psi \sim \psi'$ such that $\phi' \circ \psi'$ is defined.

For any subgroup $H\leq G$, the {\em relative commensurator} of $H$ in $G$ is defined by
\[
\Comm_G(H) = \{ g\in G \mid [H : g H g^{-1} \cap
H] < \infty \text{ and } [g H g^{-1} : g H g^{-1} \cap
H] <\infty\}.
\]
For any group $G$, there is a homomorphism $G\to \Aut(G)$ defined by $g\mapsto \conj{g}$, where $\conj{g}(x) = g x g^{-1}$ for all $x\in G$. Let $\commincl: G\to \Comm(G)$ denote the composition $G\to \Aut(G) \to \Comm(G)$. The proof of the following lemma is straightforward.
\begin{lemma}
\LabLem{universal-comm} For any group $G$,
\[
\Comm_{\Comm(G)}(\commincl(G)) = \Comm(G).
\]
\end{lemma}

\subsection{Commensurations as pro-automorphisms}

For any abstract group $G$, let $\ficol{G}$ denote the collection of finite-index subgroups of $G$, equipped with inclusions $G_1\to G_2$ whenever $G_1\fidx G_2 \fidx G$. The intersection of any two finite-index subgroups has finite index in $G$, so $\ficol{G}$ is the collection of groups of an inverse system $\bfG$ in the category $\Grp$ when ordered by reverse inclusion.

\begin{proposition}
\LabProp{comm-is-pro-aut}
    Let $G$ be a group and $\bfG$ be the inverse system of finite-index subgroups of $G$. Then there is an isomorphism
    \[
    \zeta: \Comm(G) \cong \Aut_{\pro{\Grp}}(\bfG).
    \]
\end{proposition}
\begin{proof}
For any $H\fidx G$, the system of finite-index subgroups of $H$ is cofinal in the system of finite-index subgroups of $G$. By \RefLem{lem:cofinal-system}, the inverse of the restriction map is a $\pro{\Grp}$ isomorphism $i_{GH} : \mathbf{H} \to \mathbf{G}$. These isomorphisms are functorial; if $H_1\fidx H_2\fidx G$ then $i_{G H} = i_{G H_2} \circ i_{H_2 H_1}$.

Suppose $H,K\fidx G$ and $\phi: H\to K$ is an isomorphism. The map $\phi$ induces a bijection between $\ficol{H}$ and $\ficol{K}$. The collection of isomorphisms $\left\{\restr{\phi}{H'}  \suchthat H' \fidx H \right\}$ is a morphism of inverse systems, 
which induces a $\pro{\Grp}$ isomorphism $\phi_\ast : \mathbf{H} \to \mathbf{K}$.

It is an exercise to verify two properties. First, the assignment $\phi\mapsto \phi_*$ is functorial in the sense that $(\phi\circ \psi)_* = \phi_* \circ \psi_*$ whenever the composition $\phi\circ \psi$ is defined. Second, if $\psi: H' \to K'$ is the restriction of some isomorphism $\phi: H\to K$ to a domain $H' \fidx H$, then $\phi_* = i_{K K'} \circ \psi_* \circ i_{H H'}^{-1}$.

We now define $\zeta$. Given any $[\phi]\in \Comm(G)$ represented by an isomorphism $\phi: H\to K$, let $\zeta([\phi]) = i_{G K} \circ \phi_\ast \circ i_{G H}^{-1}$. To see that $\zeta$ is well-defined, suppose $\phi:H\to K$ is an isomorphism and $\psi: H' \to K'$ is the restriction of $\phi$ to some $H'\fidx H$. Then by the above observations we have
\begin{align*}
i_{G K'} \circ \psi_\ast \circ i_{G H'}^{-1} &= (i_{G K} \circ i_{K K'}) \circ \psi_\ast \circ (i_{G H} \circ i_{H H'})^{-1} \\
&= i_{G K} \circ (i_{K K'} \circ \psi_* \circ i_{H H'}^{-1}) \circ i_{G H}^{-1}\\
&= i_{G K} \circ \phi_* \circ i_{G H}^{-1}
\end{align*}
The fact that $\zeta$ is a homomorphism follows from functoriality of the assignment $\phi \mapsto \phi_*$.

To see that $\zeta$ is injective, suppose a partial automorphism $\phi \from H \to K$ satisfies $\zeta( [\phi] ) = \mathbf{id} \in \Aut_{\pro{\Grp}}(\bfG)$. This implies that $\phi_\ast \circ i_{G H}^{-1}$  and $i_{G K}^{-1}$ are equivalent maps of inverse systems $\mathbf{G} \to \mathbf{K}$. Because $i_{G H}^{-1}$ and $i_{G K}^{-1}$ are restriction maps, by definition of equivalence there is a finite-index subgroup $H_0 \fidx H$ such that $\restr{\phi}{H_0} = \id_{H_0}$. Therefore $[\phi]$ is trivial in $\Comm(G)$.

To see that $\zeta$ is surjective, suppose $[\mathbf{f}] \in \Aut_{\pro{\Grp}} (\bfG)$ is represented by a collection of homomorphisms $\{ f_{\lambda} : G_{\rho(\lambda)} \to G_\lambda \mid G_\lambda \in \ficol{G}\}$, where $\rho : \Lambda \to \Lambda$ is a function of the indexing set of $\ficol{G}$. Because $[\mathbf{f}]$ is an isomorphism, it is both a monomorphism and an epimorphism \cite{mardesic-segal}*{Ch II, \S2.2, Theorem 6}. Because $[\mathbf{f}]$ is a monomorphism, there is an index $\lambda_0 \in \Lambda$ and a subgroup $H_0 \fidx G_{\rho(\lambda_0)}$ such that $\restr{f_{\lambda_0}}{H_0}$ is injective \cite{mardesic-segal}*{Ch II, \S2.1, Theorem 2}. Because $[\mathbf{f}]$ is an epimorphism, there is a finite-index subgroup $K_0 \fidx G$ such that $K_0\fidx f_{\lambda_0}(H_0)$ \cite{mardesic-segal}*{Ch II, \S2.1, Theorem 4}. Let $H = H_0$ and $K = f_{\lambda_0}(H)$. Then $f_0 = \restr{f_{\lambda_0}}{H} : H \to K$ is an isomorphism between finite-index subgroups of $G$.

It remains to check that $[i_{G K} \circ (f_0)_\ast \circ i_{G H}^{-1}] = [\mathbf{f}]$. Let $f' = i_{G K} \circ (f_0)_\ast \circ i_{G H}^{-1}$ be the endomorphism of the inverse system $\ficol{G}$, with associated indexing function $\rho' : \Lambda \to \Lambda$. By definition, for any $\lambda \in \Lambda$, we have $G_{\rho'(\lambda)} = f_{\lambda_0}^{-1}(G_\lambda\cap K)$ and $f'_\lambda = f_{\lambda_0}|_{G_{\rho'(\lambda)}}$.

Consider any $G_\lambda \in \ficol{G}$ and let $\lambda' \in \Lambda$ be the index such that $G_{\lambda'} = G_\lambda \cap K$. Because $G_{\lambda'}\fidx G_\lambda$, the fact that $f$ is a morphism of inverse systems implies that there is a subgroup $S_1 \fidx G$ such that $\restr{f_\lambda}{S_1} = \restr{f_{\lambda'}}{S_1}$. Similarly, because $G_{\lambda'} \leq G_{\lambda_0}$, there is some $S_2 \fidx G$ such that $\restr{f_{\lambda'}}{S_2} = \restr{f_{\lambda_0}}{S_2}$. Now let $D = S_1 \cap S_2 \cap G_{\rho'(\lambda)} \fidx G$. Combining the above observations we have $\restr{f_\lambda}{D} = \restr{f'_\lambda}{D}$, which completes the proof.
\end{proof}

\section{Full solenoids over unloopable spaces} \LabSec{full-solenoids}

The limit space of a sequence of finite-sheeted, regular covering maps appeared early in the study of homogeneous topological spaces as a source of examples and non-examples. McCord~\cite{mccord} gave a general account of the structure of such a space, not requiring coverings to be finite-sheeted, which he called a \emph{solenoidal space}. Elsewhere in the literature these spaces have simply been called \emph{solenoids}, while some authors reserve the name solenoid for inverse limits of systems of finite-sheeted
coverings of closed manifolds, or of only the circle (the original solenoids of van Dantzig and Vietoris~\cites{vanDantzig1930,vietoris}). Limits of sequences of finite-sheeted covers that are not necessarily regular are often called {\em weak solenoids}~\cite{rogers-tollefson}.
We are interested in the limit of the inverse system of \emph{all} finite-sheeted covers of a space, which we call the \emph{full solenoid} (see \RefDef{full-solenoid}). After introducing the definitions with care we review McCord's work in this setting.

\begin{definition}
	A topological space $X$ is \emph{unloopable} if it is path-connected, locally path-connected, and semi-locally simply connected.\footnote{McCord calls such spaces {\em nice} \cite{mccord}. We borrow our terminology from Bourbaki, which uses the term {\em d\`{e}la\c{c}ables}.}  Equivalently, $X$ is unloopable if it has a locally path-connected universal cover.
\end{definition}

We will typically work in categories whose objects are pointed topological spaces. When there is no chance for confusion, we will use the same symbol to denote the basepoints in different spaces.

\begin{convention} 
When working with an unloopable pointed topological space $(X,\ast)$, 
we fix a pointed universal cover $(\ucov{X},\ast) \to (X,\ast)$ once and
for all. There is a left action $\pi_1(X,\ast) \times \ucov{X} \to \ucov{X}$ written $(g,x)\mapsto gx$. We then realize the Galois correspondence between connected covers of $(X,\ast)$ and subgroups of its fundamental group by setting $(X_H, \ast) = H \backslash (\ucov{X},\ast)$ for a subgroup $H \leq \pi_1(X, \ast)$. We will 
refer to constructions as {\em canonical} if they are canonical for a given
choice of universal cover.
\end{convention}

\begin{definition}
\LabDef{full-solenoid}
	Given an unloopable pointed topological space $(X,\ast)$, the \emph{full solenoid} over
	$(X,\ast)$, denoted $(\sol{X},\ast)$, is the inverse limit of the system of connected pointed finite-sheeted
	covers of $(X,\ast)$ in $\Top_\ast$:
		\[ (\sol{X},\ast) = \varprojlim_{H \fidx \pi_1(X,\ast)} (X_H,\ast). \]
	The path-components of $(\sol{X}, \ast)$ are called \emph{leaves}.
\end{definition}

Given an unloopable pointed space $(X,\ast)$ with full solenoid $(\sol{X},\ast)$, let $(\mathbf{X}, \ast)$ denote the inverse system of finite-sheeted covers of $(X,\ast)$ with respect to a given universal cover. For any finite-index subgroup $H\fidx \pi_1(X,\ast)$ there is a projection map $p_H \from (\sol{X},\ast) \to (X_H,\ast)$. The collection of all such $p_H$ determines a map of inverse systems, hence a morphism $\mathbf{p} \from (\sol{X},\ast) \to (\mathbf{X}, \ast)$ in $\pro{\Top_*}$.

The construction of $(\sol{X},\ast)$ is independent of choice of universal cover in the following senses. First, any two universal covers of an unloopable pointed space $(X,\ast)$ are homeomorphic by a unique morphism of pointed covering spaces. Any such homeomorphism induces homeomorphisms between covering spaces $(X_H,\ast)$ compatible with the covering maps. It follows that $(\sol{X},\ast)$ is well-defined up to homeomorphism commuting with the projection $(\sol{X},\ast)\to (X,\ast)$.
Second, for a given choice of universal cover, the collection $\left\{ (X_H,\ast) \suchthat H\fidx \pi_1(X,\ast) \right\}$ contains exactly one representative from each equivalence class of connected, finite-sheeted pointed cover of $(X,\ast)$. Therefore the full solenoid $(\sol{X},\ast)$ is homeomorphic to the inverse limit of the collection of {\em all} connected, finite-sheeted pointed covers of $(X,\ast)$.

McCord \cite{mccord} described the structure of inverse limits of {\em sequences} of regular covers of an unloopable base space. The results of his original paper apply to inverse limits of general directed systems of regular covers of an unloopable base space with essentially no change to the arguments.
The system of finite-sheeted, regular covers of $(X, \ast)$ is cofinal in the system of finite-sheeted covers of $(X, \ast)$, so McCord's results apply to full solenoids over unloopable spaces.

\begin{definition}
Given an unloopable pointed topological space $(X,\ast)$, the \emph{baseleaf} of $(\sol{X},\ast)$ is the image of the canonical map $\blmap \from (\ucov{X},\ast) \to (\sol{X},\ast)$ induced by the covering maps $\ucov{X} \to X_H$ for $H\fidx \pi_1(X,\ast)$.
\end{definition}

For any unloopable pointed topological space $(X,\ast)$, the collection of finite quotients $\pi_1(X,\ast) / H$ for $H\nfidx \pi_1(X,\ast)$ forms an inverse system in $\Grp$ under the quotient maps $\pi_1(X,\ast) / H \to \pi_1(X,\ast) / K$ whenever $H\leq K$. The profinite completion is
	\[\profin{\pi_1(X,\ast)} = \varprojlim_{H\nfidx \pi_1(X,\ast)} \pi_1(X,\ast) / H.\]  
The actions $\pi_1(X,\ast)/H \curvearrowright X_H$ determine an effective topological group action $\profin{\pi_1(X,\ast)} \curvearrowright \sol{X}$ \cite{mccord}*{Lemma 5.2}. 

There is a continuous surjective map $\Pi \from \sol{\pi_1(X,\ast)} \times \ucov{X} \to \sol{X}$ defined by $\Pi(t,x) = t \blmap(x)$, where $\blmap$ is the canonical baseleaf map. There is a continuous left action of $\pi_1(X,\ast)$ on 
$\profin{\pi_1(X,\ast)} \times \ucov{X}$ defined by 
\[
g\cdot (\gamma, x) = (\gamma\pfi(g)^{-1}, gx),
\]
where $\pfi \from \pi_1(X,\ast) \to \profin{\pi_1(X,\ast)}$ is the canonical 
map. McCord proves that $\Pi \from \sol{\pi_1(X,\ast)} \times \ucov{X} \to \sol{X}$ is a (generalized) covering map whose covering transformations are precisely those homeomorphisms given by the action of $\pi_1(X,\ast)$. 

Let $\solmod{\pi_1(X,\ast)}{X}$ denote the quotient of $\sol{\pi_1(X,\ast)} \times \ucov{X}$ by the action of $\pi_1(X,\ast)$. We summarize the consequences of the above in the following theorem.

\begin{theorem}[\cite{mccord}*{Theorems 5.5, 5.6, 5.8, 5.12}]
\LabThm{metric-model-new}
	Suppose $(X,\ast)$ is an unloopable pointed topological space. 
	\begin{enumerate}
	\item There is a homeomorphism 
		\[ \left( \solmod{\pi_1(X,\ast)}{X}, [\id, \ast] \right) \cong (\sol{X}, \ast). \] 
	The canonical
	baseleaf map $\blmap \from \ucov{X} \to \sol{X}$ is the composition of $\ucov{X}
	\to \{ \id \} \times \ucov{X}$ with the quotient.
	\item Let $K = \blmap(\ucov{X})$ be the baseleaf. A subset of $\sol{X}$ is a path component if and only if it is of the form $\gamma K$ for some $\gamma \in \sol{\pi_1(X,\ast)}$.
	\item The baseleaf map $\blmap : \ucov{X} \to \sol{X}$ descends to a map $N\backslash \ucov{X} \to \sol{X}$ which is a bijection onto its image, where $N = \bigcap_{H\nfidx \pi_1(X,\ast)} H$ is the residual finiteness kernel of $\pi_1(X,\ast)$.
	\item The canonical projection $(\sol{X},\ast) \to (X,\ast)$ is a principal $\profin{\pi_1(X,\ast)}$-bundle.
	\end{enumerate}
\end{theorem}

\begin{corollary}
\LabCor{baseleaf-inj}
Suppose $(X,\ast)$ is an unloopable pointed topological space with residually finite fundamental group. The baseleaf map $\blmap \from \ucov{X} \to \sol{X}$ is injective with dense image.
\end{corollary}

\section{The shape of solenoids}\LabSec{shape-of-solenoids}

Our first main result, \RefThm{ptd-homotopy-comm}, is proved using the language of shape theory.
The natural approach to shape for this setting is that of inverse systems, as described by Marde\v{s}i\'{c} and Segal~\cite{mardesic-segal}. The inverse system approach to shape associates certain inverse systems, known as expansions, to an object in a category.

\begin{definition}[\cite{mardesic-segal}*{Ch I. \S2.1, p.~19}]
    Given a category $\mathcal{T}$ and an object $X$ in $\mathcal{T}$, a
    $\pro{\mathcal{T}}$ morphism $\mathbf{p}\from X\to \mathbf{X}$ is a \emph{$\mathcal{T}$-expansion}\footnote{Marde\v{s}i\'{c} and Segal also treat the more general notion of a $\mathcal{P}$-expansion for a full subcategory $\mathcal{P}$ of $\mathcal{T}$. We only need to consider $\mathcal{T}$-expansions.}
    of $X$ if $\mathbf{X}$ is an object of $\pro{\mathcal{T}}$ and $\mathbf{p}$ satisfies a universal property: for any pro-$\mathcal{T}$ morphism $\mathbf{h}: X\to\mathbf{Y}$ with $\mathbf{Y}$ in pro-$\mathcal{T}$ there is a unique pro-$\mathcal{T}$ morphism $\mathbf{f}$ making the following diagram commute:
    \begin{center}
    \begin{tikzcd}
    X \ar{r}{\mathbf{h}} \ar[swap]{d}{\mathbf{p}}& \mathbf{Y} \\
    \mathbf{X} \ar[dashed,swap]{ur}{\mathbf{f}} &
    \end{tikzcd}
    \end{center}
\end{definition}

As we are primarily interested in morphisms, the following representation property of expansions is central.

\begin{fact}[\cite{mardesic-segal}*{Ch I. \S2.3, p.~25}]
\label{fact:propexp}
    Given $\mathcal{T}$-expansions $\mathbf{p}\from X\to \mathbf{X}$ and $\mathbf{q}\from Y\to \mathbf{Y}$
    and a $\mathcal{T}$ morphism $f\from X\to Y$, there exists a unique pro-$\mathcal{T}$ morphism $\mathbf{f}$
    such that the following square commutes in $\pro{\mathcal{T}}$
    \begin{center}
    \begin{tikzcd}
    X\ar{r}{f}\ar[swap]{d}{\mathbf{p}} & Y \ar{d}{\mathbf{q}} \\
    \mathbf{X}\ar{r}{\mathbf{f}} & \mathbf{Y}
    \end{tikzcd}
    \end{center}
\end{fact}

It follows from uniqueness that the  function
\[ \mathcal{T}(X, Y) \to \pro{\mathcal{T}}(\mathbf{X},\mathbf{Y}) \]
is natural in the sense that it respects composition; if $f\circ g = h$ in $\mathcal{T}$, then $\mathbf{f} \circ \mathbf{g} = \mathbf{h}$ in $\pro{\mathcal{T}}$.
In general, though, the assignment may be neither injective nor surjective~\cite{mardesic-segal}*{Ch I. \S2.3, Remark 9}.

Specializing to full solenoids over unloopable pointed compact Hausdorff spaces, we are able to apply these facts 
to study self-homotopy equivalences in light of the following theorem.

\begin{theorem}[\cite{mardesic-segal}*{Ch I. \S5.4, Theorem 13}]\LabThm{cpt-expansion}
    Suppose $(\mathbf{X},\ast)$ is an inverse system of pointed compact Hausdorff spaces. If $(X_\infty, \ast)$ is the inverse limit of $(\mathbf{X},\ast)$ in $\Cpt_\ast$, then $\mathbf{p} \from (X_\infty,\ast) \to (\mathbf{X},\ast)$ is an $\HTop_\ast$-expansion
    of $(X_\infty, \ast)$.
\end{theorem}

The proof of \RefThm{ptd-homotopy-comm} relies on the following proposition. Compare, for example, to improvability results of Geoghegan and Krasinkiewicz \cite{geoghegan-krasinkiewicz}.

\begin{proposition}
\label{prop:nat-pro-bijection}
Suppose $(X,\ast)$ and $(Y,\star)$ are unloopable pointed compact Hausdorff spaces with full solenoids $(\sol{X}, \ast)$ and $(\sol{Y}, \star)$. Let $(\mathbf{X},\ast)$ and $(\mathbf{Y},\star)$ be the inverse systems of connected, finite-sheeted pointed covers. Then there is a natural bijection
\[ \HCpt_\ast(\sol{X},\sol{Y}) \to \pro{\HCpt}_\ast(\mathbf{X},\mathbf{Y}) \]
\end{proposition}

\begin{lemma} \LabLem{covering-homotopy-limit}
Suppose $(X,\ast)$ and $(Y,\star)$ are  unloopable pointed compact Hausdorff spaces. Let $(\mathbf{X},\ast)$ and $(\mathbf{Y},\star)$ be the inverse systems of connected, finite-sheeted pointed covers. Then for every
morphism $[\mathbf{f}] \in \pro{\HCpt}_\ast(\mathbf{X},\mathbf{Y})$ there is a representative $\mathbf{g}$ that
has a limit in $\Cpt_\ast$
\end{lemma}

\begin{proof}
Let the inverse system $\mathbf{X}$ be indexed by $\Lambda$ with bonding morphisms $p_{\lambda_1\lambda_2} \from X_{\lambda_1} \to X_{\lambda_2}$, and let $\mathbf{Y}$ be indexed by $M$ with bonding morphisms $q_{\mu_1\mu_2} \from Y_{\mu_1} \to Y_{\mu_2}$. Both $\Lambda$ and $M$ are cofinite directed sets with unique minimum elements. Let $\minelt\in M$ denote the minimum index, so that $Y_\minelt = Y$.

By \RefLem{lem:cofinite-morphisms}, any given morphism $[\mathbf{f}] \in \pro{\HCpt}_\ast(\mathbf{X},\mathbf{Y})$ is represented by an order-preserving function $\phi \from M \to \Lambda$ and a collection of continuous maps $f_\mu \from (X_{\phi(\mu)},\ast) \to (Y_\mu,\star)$ such that if $\mu \leq \mu'$ then there is a pointed homotopy $f_\mu \circ p_{\phi(\mu)\phi(\mu')} \sim q_{\mu\mu'}\circ f_{\mu'}$. To construct the desired morphism of inverse systems, let $g_\minelt = f_\minelt \from X_{\phi(\minelt)} \to Y_\minelt$. Given any $\mu \geq \minelt$, the existence of the pointed homotopy $f_\minelt \circ p_{\phi(\minelt)\phi(\mu)} \sim q_{\minelt\mu}\circ f_{\mu}$ guarantees that $f_\minelt \circ p_{\phi(\minelt)\phi(\mu)}$ uniquely lifts to a map $g_{\mu} \from X_{\phi(\mu)} \to Y_\mu$ satisfying the equality of continuous maps $f_\minelt \circ p_{\phi(\minelt)\phi(\mu)} = q_{\minelt\mu} \circ g_\mu$.

By construction, for any $\mu\in M$ we we have pointed homotopies
\[
 q_{\minelt\mu}\circ f_{\mu} \sim f_\minelt \circ p_{\phi(\minelt)\phi(\mu)} \sim q_{\minelt\mu} \circ g_\mu.
\]
The homotopy lifting property provides a pointed homotopy $f_\mu \sim g_\mu$. It follows that the collection of maps $(g_\mu)$ defines a map $\mathbf{g}$ of inverse systems in $\HCpt_*$ and $[\mathbf{g}] = [\mathbf{f}] \in \pro{\HCpt}_\ast(\mathbf{X},\mathbf{Y})$.

In fact, it follows from uniqueness of liftings that we may consider $\mathbf{g}$ as a morphism of systems in $\Cpt_*$. To see this, consider indices $\mu\leq \mu' \in M$. Then 
\[
    q_{\minelt \mu'} \circ g_{\mu'} = f_\minelt \circ p_{\phi(\minelt)\phi(\mu')} = (f_\minelt \circ p_{\phi(\minelt)\phi(\mu)}) \circ p_{\phi(\mu)\phi(\mu')} = q_{\minelt \mu} \circ (g_\mu \circ p_{\phi(\mu)\phi(\mu')})
\]
while on the other hand
\[
    q_{\minelt \mu'} \circ g_{\mu'} = q_{\minelt \mu} \circ (q_{\mu \mu'} \circ g_{\mu'}).
\]
Then uniqueness of lifts implies $g_\mu \circ p_{\phi(\mu)\phi(\mu')} = q_{\mu \mu'} \circ g_{\mu'}$. 

For each $\lambda\in \Lambda$, let $p_\lambda \from \sol{X} \to X_\lambda$ be the the system of projections. For each $\mu$, define $\sol{g_\mu} \from \sol{X} \to Y_\mu$ by $\sol{g_\mu} = g_{\phi(\mu)} \circ p_{\phi(\mu)}$. These determine a continuous function $\sol{g} : \sol{X} \to \sol{Y}$, the limit of $\mathbf{g}$.
\end{proof}

\begin{proof}[Proof of \cref{prop:nat-pro-bijection}]
By \RefThm{cpt-expansion}, both $(\sol X,\ast) \to (\mathbf{X},\ast)$ and $(\sol Y, \star) \to (\mathbf{Y},\star)$ are  $\HTop_\ast$-expansions. It follows from \cref{fact:propexp} that there is a natural function
\[ \HCpt_\ast(\sol{X},\sol{Y}) \to \pro{\HTop}_\ast(\mathbf{X},\mathbf{Y}) \]
Because $(\mathbf{X}, \ast)$ and $(\mathbf{Y},\star)$ are systems of compact Hausdorff spaces and $\HCpt_\ast$ is a full subcategory of $\HTop_\ast$, this determines a natural function
\[ \HCpt_\ast(\sol{X},\sol{Y}) \to \pro{\HCpt}_\ast(\mathbf{X},\mathbf{Y}) \]
By \RefLem{covering-homotopy-limit}, this map is surjective. 

Now suppose $f,g \from (\sol{X},\ast) \to (\sol{Y},\star)$ are pointed continuous maps whose pointed homotopy classes map to the same element $[\mathbf{h}]$ of $\pro{\HCpt}_\ast(\mathbf{X},\mathbf{Y})$. Let $[\mathbf{q}] \from \sol{Y} \to \mathbf{Y}$ be the expansion morphism. Since $f$ and $g$ both map to $\mathbf{h}$, by Fact~\ref{fact:propexp},
the morphisms $\mathbf{q}\circ f$ and $\mathbf{q}\circ g$ are representatives of the same element of $\pro{\HCpt}_\ast(\sol{X},\mathbf{Y})$. Consequently the maps $q\circ f$ and $q\circ g$ are pointed homotopic as maps $(\sol{X}, \ast) \to (Y,\star)$, where $q\from \sol{Y}\to Y$ is the projection.

By \RefThm{metric-model-new}, the map $q$ is a fiber bundle over a compact Hausdorff space and therefore a fibration \cite{spanier}*{\S2.7, Cor 14}. Therefore the pointed homotopy $q\circ f \sim q\circ g$ lifts to a homotopy $f\sim F$ for some $F \from \sol{X} \to \sol{Y}$. Because fibers of $q$ are totally disconnected, $q$ has the unique path lifting property \cite{spanier}*{\S2.2, Thm 5}. This implies that both $F$ and the homotopy $f\sim F$ are pointed. Moreover, because the baseleaf of $\sol{X}$ is path-connected, $F$ and $g$ agree on the baseleaf of $\sol{X}$. Indeed, if $\gamma$ is a based path of $\ucov{X}$ starting at the basepoint, then $q\circ F(\gamma) = q\circ g(\gamma)$, so that $F(\gamma)$ is the unique lift
of $q\circ g(\gamma)$, and therefore equal to $g(\gamma)$. But the baseleaf is dense in $\sol{X}$, so we conclude that $F = g$ as pointed continuous maps.
\end{proof}

Combining these facts we arrive at a topological description of the abstract commensurator of $\pi_1(X,\ast)$. For any pointed topological space $(Y,\star)$, let $\HMod(Y,\star)$ be the automorphism group of $(Y,\star)$ in $\Top_\ast$. Elements of $\HMod(Y,\star)$ are equivalence classes of pointed homotopy equivalences, modulo pointed homotopy.

\begin{theorem}
\LabThm{ptd-homotopy-comm}
    Suppose $(X,\ast)$ is an unloopable pointed compact Hausdorff space. Then there is a  homomorphism
    \[ C \from \HMod(\sol{X},\ast) \to \Comm(\pi_1(X,\ast)) \]
    If $X$ is homotopy equivalent to an aspherical CW complex, this map is an isomorphism.
\end{theorem}

\begin{proof}
Let $(\mathbf{X}, \ast)$ be the full system of finite-sheeted pointed covers of $(X,\ast)$. By \cref{prop:nat-pro-bijection}, there is a natural bijection
\[ \HCpt_\ast(\sol{X},\sol{X}) \to
\pro{\HCpt}_\ast(\mathbf{X},\mathbf{X}) \]

A standard characterization of pro-morphisms is the natural bijection~\cite{mardesic-segal}*{Ch I. \S1.1, Remark 4, p.~8},
    \[ \pro{\HCpt}_\ast((\mathbf{X},\ast),(\mathbf{X},\ast)) \simeq \varprojlim_\alpha \varinjlim_\beta \HCpt_\ast(X_\alpha, X_\beta). \]
    Since $\pi_1$ is functorial, there is a natural function
    \[ \varprojlim_\alpha\varinjlim_\beta \HCpt_\ast(X_\alpha, X_\beta) \to \varprojlim_\alpha\varinjlim_\beta \Grp (\pi_1(X_\alpha,\ast),\pi_1(X_\beta,\ast)). \]
    When $X$ is homotopy equivalent to an aspherical CW complex, each $X_\alpha$ is also homotopy equivalent to an aspherical CW complex, and so this function is
    a term-by-term bijection by the Whitehead theorem.

    Let $\pi_1(\mathbf{X}, \ast)$ be the pro-group obtained by applying the $\pi_1$ functor to the inverse system $(\mathbf{X}, \ast)$. As noted above, there is a natural bijection
    \[  \varprojlim_\alpha\varinjlim_\beta\Grp (\pi_1(X_\alpha,\ast),\pi_1(X_\beta,\ast)) \simeq \pro{\Grp}((\pi_1(\mathbf{X},\ast)),(\pi_1(\mathbf{X},\ast)) ) \]
    Combining these natural morphisms, restricting to the automorphism groups in each category, and applying \RefProp{comm-is-pro-aut}, we arrive at
    a homomorphism of groups
    \[ \HMod(\sol{X},\ast) \cong \Aut_{\pro{\HCpt}_\ast}(\mathbf{X},\ast) \to \Aut_{\pro{\Grp}}(\pi_1( \mathbf{X},\ast)) \cong \Comm(\pi_1(X,\ast)) \]
    that is an isomorphism when $X$ is homotopy equivalent to an aspherical CW complex.
\end{proof}

\begin{remark}
Suppose $(X,\ast)$ is an unloopable pointed compact Hausdorff space homotopy equivalent to an aspherical CW complex. Let $(\mathbf{X}, \ast)$ be the system of finite covers and $(\sol{X},\ast)$ be the full solenoid over $(X,\ast)$. Then $\pi_1(\mathbf{X},\ast)$ is equal to the pro-fundamental group $\operatorname{pro-\pi_1}(\sol{X},\ast)$ \cite{mardesic-segal}*{Ch II, \S3.3}. Therefore \RefThm{ptd-homotopy-comm} provides an isomorphism
\[
\HMod(\sol{X},\ast) \cong \Aut(\operatorname{pro-\pi_1}(\sol{X},\ast)).
\]
\end{remark}

\section{Metric notions} \LabSec{metric-notions}

We now turn to the task of relating the topology of the full solenoid $(\sol{X},\ast)$ to the geometry of the group $\pi_1(X,\ast)$ in the case that $X$ is an unloopable compact geodesic metric space and $\pi_1(X,\ast)$ is residually finite. In this section we summarize the basic results used to metrize $\sol{X}$ in this case.

\subsection{Metric spaces}

When $X$ is an unloopable length metric space with metric $d_X$, $X$
determines a metric on $\ucov{X}$, which will also be denoted $d_X$.
The metric on $\ucov{X}$ is also a length metric~\cite{bridson-haefliger}*{Proposition I.3.25}, and
the covering projection is a local isometry with respect to this metric. If $X$ is additionally compact,
then it follows from the metric Hopf-Rinow theorem~\cite{bridson-haefliger}*{Proposition I.3.7}
that both $X$ and $\ucov{X}$ are proper geodesic metric spaces. Moreover $\pi_1(X)$ is finitely 
presented and every finitely presented group arises this way~\cite{bridson-haefliger}*{Corollary I.8.11}. 
	
\subsection{Quasi-isometries}
For metric spaces $(X,d_X)$ and $(Y,d_Y)$, a map $f\from X\to Y$ is a {\em quasi-isometry} if there are constants $K\geq 1$ and $C > 0$ such that 
\[
 (1/K) d_X(x_1, x_2) - C \leq d_Y( f(x_1), f(x_2)) \leq K d_X(x_1,x_2) + C
\]
for any $x_1,x_2\in X$, and for any $y\in Y$ there is some $x\in X$ such that $d_Y( f(x), y) \leq C$. The {\em quasi-isometry} group of $X$ is the group $\QI(X)$ of equivalence classes of
quasi-isometries $f\from X\to X$, where $f_1$
and $f_2$ are equivalent if there is some $D\geq 0$ so that $d(f_1(x), f_2(x)) \leq D$ for all $x\in X$.
	
If $G$ is a finitely generated group, let $\QI(G)$ be the quasi-isometry group of $G$ with respect to the word metric of some finite generating set. The quasi-isometry group is independent of chosen finite generating set.
There is a natural map $\Comm(G) \to \QI(G)$,
where $[\phi] \in \Comm(G)$ determines a quasi-isometry of $G$ by precomposing
with any closest-point projection from $G$ to the domain of $\phi$. Whyte proved
that this is injective, as recorded by Farb and Mosher~\cite{farbmoshersbf}*{Proposition 7.5}.

\begin{lemma}
\label{lem:baseleaf-QI-group-id}
For any unloopable compact geodesic metric space $(X,\ast)$ with fundamental group $G = \pi_1(X,\ast)$, the orbit map $G \to (\ucov{X}, \ast)$ defined by $g\mapsto g\ast$ induces an isomorphism $\QI(G) \cong \QI(\ucov{X})$.
\end{lemma}
\begin{proof}
By the Milnor-Schwarz lemma~\cite{bridson-haefliger}*{Proposition I.8.19}, the orbit map $g\mapsto g\ast$ is a quasi-isometry $G\to \ucov{X}$. This induces an isomorphism $\QI(G) \cong \QI(\ucov{X})$~\cite{bridson-haefliger}*{Exercise I.8.16.3}.
\end{proof}

In order to relate the topology of the full solenoid $(\sol{X},\ast)$ over $(X,\ast)$ to the coarse geometry of $G = \pi_1(X,\ast)$ in \RefSec{he-qi}, we will further require that the baseleaf map $\blmap \from \ucov{X} \to \sol{X}$ be an inclusion. By  \RefCor{baseleaf-inj}, this is equivalent to requiring that $G$ be residually finite.

\subsection{Profinite completions}

When a group $G$ is finitely generated, for each $n\geq 1$ there is a finite-index characteristic subgroup
\[
\Gn = \bigcap_{[G:H]\leq n} H.
\]
The subgroups $\Gn$ are a neighborhood basis of the identity for the topology on $G$ induced by the  pseudometric $\dpro$ defined by
\[
\dpro(g_1,g_2) = \exp( - \max \set{ n \suchthat g_1g_2^{-1} \in \Gn } )
\] 
if such a maximum exists, and $\dpro(g_1,g_2) = 0$ if $g_1g_2^{-1} \in H$ for all $H\fidx G$. 

If $G$ is residually finite then $\dpro$ is an ultrametric and the induced topology is Hausdorff.
In this case, the profinite completion $\profin{G}$ is the metric completion
of $G$ with respect to $\dpro$~\cite{hall}*{Theorem 3.5}, the natural homomorphism $\pfi \from G \to \profin{G}$ is an
inclusion, and $G$ acts freely by isometries on $\profin{G}$ by both left and right multiplication.

\subsection{Solenoid metric}
Now suppose $X$ is an unloopable compact geodesic metric space with residually finite fundamental group $G = \pi_1(X,\ast)$. Then $G$ is finitely presented, so by the above discussion there is a profinite metric $\dpro$ on $\profin{G}$ and and a proper, geodesic metric $d_X$ on $\ucov{X}$. Equip $\profin{G} \times \ucov{X}$ with the
	$\ell_\infty$ product metric 
	\[
	d_\infty( (\gamma_1,x_1), (\gamma_2,x_2) ) = \max \set{ \profin{d}(\gamma_1,\gamma_2), d_X(x_1,x_2)}.
	\]

Both the action of $G$ on $(\profin{G}, \dpro)$ by right multiplication
and the left action of $G$ on $(\ucov{X}, d_X)$ are isometric. Hence the left action of $G$ on
$\profin{G}\times\ucov{X}$ defined in \RefSec{full-solenoids} is isometric with respect to  $d_\infty$.
	
\begin{definition}
	Suppose $X$ is an unloopable compact geodesic metric space with residually finite fundamental group $G = \pi_1(X,\ast)$. The {\em solenoid metric} $\sigma$ on $\sol{X}$ is the metric induced from the quotient metric on $\solmod{G}{X}$, obtained by taking the infimum of $d_\infty$ over orbit
	representatives, under the identification $\sol{X} \homeo \solmod{G}{X}$.
\end{definition}

\subsection{Metric balls in a full solenoid} Continue to consider an unloopable compact geodesic metric space $X$ with residually finite fundamental group $G = \pi_1(X,\ast)$.
	\begin{definition}
The \emph{injectivity radius} of an unloopable metric space $(Y, d_Y)$, denoted $\injrad(Y)$, is the supremum of all $R$ such that for all $x\in\ucov{Y}$ the covering projection of
	$B_{d_Y}(x, R)$ is an isometry onto its image.
\end{definition}

If $X$ is an unloopable compact geodesic metric space, then $0 < \injrad(X) < \infty$. Note that for all $g\in \pi_1(X,\ast)$ and $x\in\ucov{X}$, the inequality $d_X(g\cdot x, x)
\geq \injrad(X)$ holds.

\begin{lemma}
\LabLem{small-ball-product}
Suppose $X$ is an unloopable compact geodesic metric space with residually finite fundamental group $G = \pi_1(X,\ast)$, and suppose $0 < \epsilon < \injrad(X)/4$. Then for any $\sol{x} = [ (\gamma, x)] \in \solmod{G}{X}$, there is an isometry
\[
B_\sigma( \sol{x}, \epsilon ) \cong B_{\sol{d}}(\gamma, \epsilon) \times B_{d_{\ucov{X}}}(x, \epsilon).
\]
In particular, each path component of $B_\sigma( \sol{x}, \epsilon)$ is isometrically identified with $B_{d_X}(x, \epsilon)$.
\end{lemma}

\begin{proof}
Take any $(\gamma,x) \in \profin{G} \times \ucov{X}$. Consider $B = B_{d_\infty}( (\gamma,x), \epsilon) = B_{\sol{d}}(\gamma, \epsilon) \times B_{d_X}(x, \epsilon)$. Then for any $y\in B_{d_X}(x,\epsilon)$ and any $g\in G$, we have $d_X(y, gy) \geq \injrad(X) > 4\epsilon > 2\epsilon$ and therefore $gy\notin B_{d_X}(x,\epsilon)$. It follows that the ball $B$ maps injectively into $\sol{X}$.

To see that the inclusion $B\to \sol{X}$ is isometric onto its image, it suffices to show that if $y,z \in B_{d_X}(x,\epsilon)$ and $g\in G$, then $d_X(y,gz) \geq 2\epsilon$. This follows from the reverse triangle inequality:
\[
d_X(y,gz) \geq d_X(z,gz) - d_X(y,z) \geq \injrad(X) - 2\epsilon > 2\epsilon.   \qedhere
\]
\end{proof}

\section{Homotopy and baseleaf quasi-isometry} \LabSec{he-qi}

In general, compactness improves the behavior of continuous functions between
metric spaces. In the setting of full solenoids over spaces with residually finite fundamental group, compactness allows
one to translate from the topology of the solenoid to the coarse
geometry of the baseleaf. This has been studied in the context of foliated manifolds, for example in work of Hurder \cite{hurder-coarse}*{\S2.4}.

\begin{definition}
	A function $f\from (X, d_X) \to (Y, d_Y)$ of metric spaces is
	\emph{$(L, C)$-coarsely Lipschitz} if for all $x, y\in X$, if
	for all $x, y \in X$
		\[ d_Y(f(x), f(y)) \leq L d_X(x,y) + C. \]
\end{definition}

Given an unloopable pointed space $(X,\ast)$, any pointed continuous map $f\from (\sol{X}, \ast) \to (\sol{X},\ast)$ maps the baseleaf into itself. When $\pi_1(X,\ast)$ is residually finite, the baseleaf map $\blmap \from \ucov{X} \to \sol{X}$ is a bijection onto its image by \RefCor{baseleaf-inj}. Identifying $\ucov{X}$ with its image in $\sol{X}$, say that $f$ restricts to a function $f_{\ucov{X}} \from (\ucov{X}, \ast) \to (\ucov{X},\ast)$.

\begin{lemma}
\LabLem{uniform-continuous-coarse-lipschitz}
Suppose $(X,\ast)$ is an unloopable compact geodesic metric
	space with residually finite fundamental group $G = \pi_1(X,\ast)$. Suppose $f\from (\sol{X},\ast) \to (\sol{X},\ast)$ is a pointed
	continuous function. Then for any $C > 0$ there exists
	$\blip_C(f)>0$ such that $f_{\ucov{X}} \from (\ucov{X}, d_X) \to (\ucov{X},
	d_X)$ is $(\blip_C(f), C)$-coarsely Lipschitz.
\end{lemma}

\begin{proof}
Since $X$ is compact and Hausdorff,
	$\solmod{G}{X}$ is compact. Therefore, since $f$ is continuous it is
	uniformly continuous in the $\dsol$ metric on $\sol{X}$.

Take $0 < \epsilon < \min\set{\injrad(X)/4, C}$, and choose $0 < \delta <
	\epsilon$ such that for all $x, y\in \sol{X}$ if $\dsol(x, y) <
	\delta$ then $\dsol(f(x), f(y)) < \epsilon$. Now suppose $x, y\in
	\ucov{X}$ and $d_X(x, y) < \delta$. This implies $\dsol(x, y) < \delta$
	by definition. Let $\gamma$ be a $d_X$ geodesic joining $x$ to $y$.
	Since $\gamma$ is a geodesic, $\gamma \subset B_{d_X}(x,\delta)\subseteq B_{\dsol}(x,\delta)$.
	Therefore by uniform continuity,
	$f(\gamma) \subseteq B_{\dsol}(f(x),\epsilon)$, so $f(x)$ and $f(y)$
	are in the same path component of $B_{\dsol}(f(x),\epsilon)$. The path
	component of $B_{\dsol}(f(x),\epsilon)$ containing $f(x)$ is simply
	$B_{d_X}(f(x),\epsilon)$ by \RefLem{small-ball-product}. Therefore if $d_X(x,y) < \delta$ then
	$d_X(f(x), f(y)) < \epsilon$.

Next, suppose $x, y \in \ucov{X}$ and $M = d_X(x, y) \geq \delta$. Divide a
	geodesic between $x$ and $y$ into $N = \floor{\frac{M}{\delta} } \geq 1$
	length $\delta$ segments, with endpoints $x = t_0, t_1, \ldots, t_N =
	y$. By repeatedly applying the triangle inequality we estimate
	\[ d_X(f(x), f(y)) \leq N\epsilon \leq \frac{\epsilon}{\delta}d_X(x,y). \]

Since $\epsilon < C$, by setting $\blip_C(f) = \frac{\epsilon}{\delta}$ we
	conclude that for all $x, y \in \ucov{X}$
	\[  d_X(x, y) < \blip_C(f)d_X(x, y) + C.\qedhere\]
\end{proof}

\begin{lemma}
\LabLem{homotopy-bounded}
	Suppose $(X,\ast)$ is an unloopable compact geodesic metric space with residually finite fundamental group.
	Given continuous pointed functions $f, g\from (\sol{X},\ast) \to
	(\sol{X},\ast)$ such that $f \sim g$, there exists a constant $C$ such that
	for all $p\in \ucov{X}$,
	\[ d_X(f_{\ucov{X}}(p), g_{\ucov{X}}(p)) \leq C.\]
\end{lemma}

\begin{proof}
Take $0 < \epsilon < \injrad(X)/4$. Let $F$ 
	be a homotopy witnessing $f\sim g$. Since $\sol{X}\times[0,1]$ is
	compact, $F$ is
	uniformly continuous with respect to the $\ell_\infty$ product metric $d_\infty$. Choose $\delta > 0$ such that for all $(x,t),
	(y,s) \in \sol{X}\times [0,1]$,
	\[ d_\infty( (x,t), (y,s) ) < \delta 
		\Rightarrow \dsol( F(x,t), F(y,s)) < \epsilon. \]
	
Choose $0 = t_0 < t_1 < \cdots < t_L = 1$ in $[0,1]$ such that $|t_{n+1} -
	t_{n} | < \delta$. Given any $p\in \ucov{X}$, the image
	$F(\{p\}\times [t_n,t_{n+1}])$ is a path connected subset of
	$B_{\dsol}( F(p,t_n), \epsilon)$. By our choice of $\epsilon$, \RefLem{small-ball-product} implies $F(\{p\}\times[t_n,t_{n+1}])$ is in fact
	contained in $B_{d_X}(F(p,t_n),\epsilon)$ in $(\ucov{X}, d_X)$, thus 
		\[ d_X( F(p,t_n), F(p, t_{n+1})) < \epsilon. \]
	
Therefore, for any $p\in\ucov{X}$, repeated application of the triangle
	inequality gives
	\[ d_X(f_{\ucov{X}}(p), g_{\ucov{X}}(p))\leq L \cdot \epsilon. \qedhere\]
\end{proof}

Let $\HE(\sol{X},\ast)$ be the collection of all pointed homotopy equivalences $f \from (\sol{X}, \ast) \to (\sol{X}, \ast)$. Note that $\HE(\sol{X},\ast)$ is a monoid, and $\HMod(\sol{X}, \ast)$ is the quotient group consisting of pointed homotopy classes.

\begin{corollary}
\LabCor{fg-coarse-inverses}
Given $f, g \in \HE(\sol{X},\ast)$ such that $g$ is a homotopy inverse of $f$, there exists
an $C$ such that for all $p\in \ucov{X}$,
	\[ d_X( f\circ g(p), p ) \leq C 
		\qquad\text{and}\qquad 
	   d_X(g\circ f(p),p) \leq C. \]
\end{corollary}

Combining this corollary with \RefLem{uniform-continuous-coarse-lipschitz} we
arrive at an exact description the coarse behavior of pointed
homotopy equivalences.

\begin{lemma}
\LabLem{he-qi-pointed}
	Suppose $(X,\ast)$ is an unloopable compact geodesic metric space with residually finite fundamental group.
	For each $f\in \HE(\sol{X},\ast)$, the restriction $f_{\ucov{X}}$ is a
	quasi-isometry of $(\ucov{X}, d_X)$.
\end{lemma}

\begin{proof}
	Given any  $f\in \HE(\sol{X},\ast)$, fix a homotopy inverse $g$. By \RefCor{fg-coarse-inverses} there exists
	a $C$ such that for all $p\in \ucov{X}$, we have $d_X (f\circ g(p), p) \leq
	C$ and $d_X(g \circ f(p), p) \leq C$.

	Next, by \RefLem{uniform-continuous-coarse-lipschitz} with $L =
	\max\{\blip_C(f), \blip_C(g)\}$, both $f$ and $g$ are $(L, C)$-coarse
	Lipschitz on $(\ucov{X}, d_X)$. This is a characterization of an
	$(L, C)$-quasi-isometry  of $(\ucov{X}, d_X)$
	\cite{drutu-kapovich}*{Corollary 8.13}
	and we are done.
\end{proof}

\begin{proposition}
\LabProp{modb-qi-morphism}
	Suppose $(X,\ast)$ is an unloopable compact geodesic metric space with residually finite fundamental	group $G = \pi_1(X,\ast)$.
	There is a group homomorphism
		\[ Q\from \HMod(\sol{X},\ast) \to \QI(\ucov{X})\cong \QI(G) \]
	given by $[f] \mapsto f_{\ucov{X}}$ for a representative
	$f\in\HE(\sol{X},\ast)$.
\end{proposition}

\begin{proof}
	By \RefLem{he-qi-pointed}, given $f\in\HE(\sol{X},\ast)$,
	the restriction $f_{\ucov{X}}$ is a quasi-isometry of $\ucov{X}$.
	The assignment $f\mapsto f_{\ucov{X}}$ is a monoid homomorphism
	\[ \HE(\sol{X},\ast) \to \QI(\ucov{X}) \]

	Next, if $f, g\in\HE(\sol{X},\ast)$ and $f\sim g$, then by
	\RefLem{homotopy-bounded} there exists a constant $C$ such
	that for all $p\in \ucov{X}$, the distance $d_X(f(p), g(p)) \leq C$.
	Thus $f_{\ucov{X}}$ and $g_{\ucov{X}}$
	represent the same element of $\QI(\ucov{X})$. Therefore we obtain
	a monoid homomorphism
	\[ Q\from \HMod(\sol{X},\ast) \to \QI(\ucov{X}). \]
	Every monoid homomorphism of groups is a group
	homomorphism. The conclusion follows by composing with the isomorphism $\QI(\ucov{X}) \cong \QI(G)$ of \cref{lem:baseleaf-QI-group-id}.
\end{proof}

\begin{theorem}
\LabThm{qi-map-factors}
Suppose $(X,\ast)$ is an unloopable compact geodesic metric space with residually finite fundamental group $G = \pi_1(X,\ast)$. The homomorphism $Q\from \HMod(\sol{X},\ast)\to \QI(G)$ factors through the map $C\from \HMod(\sol{X},\ast)\to \Comm(G)$ of \RefThm{ptd-homotopy-comm}.
\end{theorem}

\begin{proof}
Suppose $[f]\in\HMod(\sol{X},\ast)$. By \RefLem{covering-homotopy-limit}, there
is an morphism of the covering system  $\{f_\alpha \from X_{\phi(\alpha)} \to X_\alpha\}$ with limit $\sol{f}$ homotopic to $f$. As in the proof of \RefThm{ptd-homotopy-comm}, the collection of induced maps $({f_\alpha}_\ast)$ is a $\pro{\Grp}$ automorphism of the inverse system of finite-index subgroups of $G$.

Pick an index $\alpha$ and set $H_0 = \pi_1(X_{\phi(\alpha)},\ast)$. As in the proof of \RefProp{comm-is-pro-aut}, there are finite-index subgroups $H,K\fidx G$ such that $\phi = \restr{{f_\alpha}_\ast}{H} \from H \to K$ is an isomorphism. By definition, $C([f]) = [\phi] \in \Comm(G)$.

The inclusion of $[\phi]$ into $\QI(G)$ is the composition of nearest point projection to $H$ and $\phi\from H\to K$. Let $(\ucov{X}, \ucobp) \to (X,\ast)$ be the universal cover of $X$. Under the identification $\QI(G) \cong \QI(\ucov{X})$ induced by the orbit map $g\mapsto g\ucobp$, the image of $[\phi]$ in $\QI(\ucov{X})$ is the map $h\ucobp \mapsto \phi(h)\ucobp$, which bijectively maps the $H$-orbit of the basepoint $\ucobp$ to the $K$-orbit of $\ucobp$, precomposed with a closest-point projection $\ucov{X} \to H\ucobp$.

Let $\ucov{f_\alpha} \from (\ucov{X}, \ucobp) \to (\ucov{X}, \ucobp)$ be the unique lift of $f_\alpha$. Since $\sol{f}$ is the limit of $f_\alpha$, for any point $h\ucobp$ in the $H$-orbit of the basepoint $\ucobp\in \ucov{X}$ we have
\[
    \sol{f}(h\ucobp) = \ucov{f_\alpha}(h\ucobp) = ({f_\alpha}_\ast(h)) \ucobp = \phi(h) \ucobp.
\]Since $H$ and $K$ are finite-index subgroups of $G$ this implies $\sol{f}$ is bounded distance from the image of $\phi$ in $\QI(\ucov{X})$. 
Therefore, by \RefLem{homotopy-bounded}, $f$ and $\phi$ are bounded distance maps, and we are done.
\end{proof}

\section{Hyperbolicity and the boundary realization} \LabSec{boundarycomm}

For a closed hyperbolic surface $\Sigma$, Odden proved the existence of an isomorphism
\[
\Comm(\pi_1(\Sigma,\ast)) \cong \Homeo(\sol{\Sigma},\ast) / \Homeo_0(\sol{\Sigma},\ast) 
\]
by proving that both groups are isomorphic to the relative commensurator of $\pi_1(\Sigma)$ in $\Homeo(\bdry \hypspace^2)$ \cite{odden}*{Theorems 4.6, 4.12}. We will prove that the abstract commensurator of any torsion-free, non-elementary  hyperbolic group $G$ is isomorphic to its relative commensurator in the group of homeomorphisms of its boundary, generalizing half of Odden's proof. We then provide a topological analogue of the other half of the proof in the case $G$ acts on a sufficiently nice metric space $X$ homotopy equivalent to a finite aspherical CW complex.

Suppose $G$ is a Gromov hyperbolic group. Its boundary $\bdry G$ is well-defined up to $G$-equivariant homeomorphism. 
There is a well-defined map $\QI(G) \to
\Homeo(\bdry G)$, which is injective in the case $\abs{\bdry G} > 3$
\cite{drutu-kapovich}*{Corollary 11.115}\footnote{Note that there is a typo
in the cited corollary. To see injectivity in the case that $G$ is non-elementary, read comments after Lemma 11.112 and in the proof of 11.130 in the same reference.}. When $\abs{\bdry G} > 3$, say $G$ is  \emph{non-elementary}. The next lemma follows from composition with the inclusion $\Comm(G)\to \QI(G)$ described in \RefSec{metric-notions}.

\begin{lemma} \LabLem{hyp-bdry-action}
Suppose $G$ is a torsion-free, non-elementary hyperbolic group. 
The composition $\Comm(G) \to \QI(G) \to \Homeo(\bdry G)$ is an injective map
\[
\Phi: \Comm(G) \to \Homeo(\bdry G).
\]
\end{lemma}

A group $G$ has the \emph{unique root property} if $g^n = h^n$ implies $g = h$
for all $g,h \in G$ and all nonzero $n\in \mathbb{Z}$. If $G$ has the unique root property and $\phi_1 \from H_1
\to K_1$ and $\phi_2\from H_2\to K_2$ are equivalent partial automorphisms of $G$, then $\phi_1$ and
$\phi_2$ agree on $H_1\cap H_2$~\cite{bartholdi-bogopolski}*{Lemma 2.4}. Note
that if $G$ satisfies the unique root property then for any $H\fidx G$ the
natural map $\Aut(H) \to  \Comm(G)$ is injective.

\begin{lemma} \LabLem{hyp-incl-comm}
Suppose $G$ is a torsion-free, non-elementary hyperbolic group. Then $G$ has the unique root property and trivial center. In particular, the map $\commincl: G\to \Comm(G)$ defined in \S\ref{sec:abcomm} is injective.
\end{lemma}
\begin{proof}
Any nontrivial element $g\in G$ in a hyperbolic group has virtually cyclic centralizer $\cent{g}{G}$ \cite{bridson-haefliger}*{Corollary III.$\Gamma$.3.10}. Since $G$ is torsion-free, $\cent{g}{G}$ is infinite cyclic. It follows that the center of $G$ is trivial, so the conjugation map $G\to \Aut(G)$ is injective. Further, it follows  that $G$ has the unique root property~\cite{bartholdi-bogopolski}*{Lemma 2.2}, and so the natural map $\Aut(G) \to \Comm(G)$ is injective.
\end{proof}

\begin{theorem}
\LabThm{hyp-boundary-comm}
Suppose $G$ is a torsion-free, non-elementary hyperbolic group.
Then the inclusion $\Phi: \Comm(G) \to \Homeo(\bdry G)$ induces an isomorphism
\[
\Comm(G) \cong \Comm_{\Homeo(\bdry G)}(G),
\]
where $G \leq \Homeo(\bdry G)$ is identified with its image under the injective map $\Phi\circ \commincl$.
\end{theorem}
\begin{proof}
By \RefLem{universal-comm} we know $\Comm(G) = \Comm_{\Comm(G)}(\commincl(G))$. It is an exercise to check that if $K$ is a group containing a subgroup $H$ and $\phi$ is any homomorphism with domain $K$, then $\phi( \Comm_K(H) ) \leq \Comm_{\phi(K)}(\phi(H))$. This implies 
\[
\Phi(\Comm(G)) \leq \Comm_{ \Phi( \Comm(G))} (\Phi(\commincl(G))) \leq \Comm_{\Homeo(\bdry G)}(G).
\]
Therefore it suffices to show that $\Phi$ has image $\Comm_{\Homeo(\bdry G)}(G)$.

Define a function $\Psi : \Comm_{\Homeo(\bdry G)}(G) \to \Comm(G)$ as follows: Given $f\in \Comm_{\Homeo(\bdry G)}(G)$, let $G_1 = f^{-1} G f \cap G$ and $G_2 = f G_1 f^{-1}$. By definition, both $G_1$ and $G_2$ are finite-index subgroups of $G$, and conjugation $\conj{f}:G_1 \to G_2$ is an isomorphism. Let $\Psi(f) = [\conj{f}] \in \Comm(G)$.

To finish the proof, it suffices to show $(\Phi\circ \Psi)(f) = f$ for any homeomorphism $f\in \Comm_{\Homeo(\bdry G)}(G)$. To this end, fix $f\in \Comm_{\Homeo(\bdry G)}(G)$ and let $\adhocf = (\Phi\circ \Psi)(f)$. Find finite-index subgroups $G_1, G_2 \leq G$ such that $f G f^{-1} = G_2$. Applying the above definitions, $\adhocf g \adhocf^{-1} = fgf^{-1}$ for any $g\in G_1$. Equivalently, for any $g\in G_1$, 
\begin{equation}
\LabEq{conj-on-bdry}
(f^{-1} \adhocf) g (f^{-1} \adhocf)^{-1} = g.
\end{equation}

We now appeal to basic facts about the action of a hyperbolic group on its boundary---see the survey of Kapovich and Benakli for a reference \cite{kapovich-benakli}*{\S4}. Every $g\in G$ has exactly two fixed points on $\bdry G$, an attracting fixed point $g^+$ and a repelling fixed point $g^-$, and attracting fixed points of elements of $G$ are dense in $\bdry G$. Suppose $x\in \bdry G$ is an attracting fixed point $g^+$ for $g\in G$. Replacing $g$ with a sufficiently large positive power, we may assuming $g\in G_1$. It follows from \RefEq{conj-on-bdry} that $f^{-1}(\adhocf(g^+))$ is a fixed point of $g$ with the same dynamical properties as $g^+$, therefore $f^{-1}(\adhocf(g^+)) = g^+$. By density of the set of attracting fixed points, this implies $\adhocf = f$. 
\end{proof}

\begin{corollary}
\LabCor{sol-hyp-bdry}
 Suppose $(X,\ast)$ is an unloopable compact geodesic metric space homotopy equivalent to an  aspherical CW complex. Suppose $G = \pi_1(X,\ast)$ is a residually finite, torsion-free, non-elementary hyperbolic group. Then the map $Q \from \HMod(\sol{X},\ast) \to \QI(G)$ of \RefProp{modb-qi-morphism} induces an isomorphism
 \[
 \HMod(\sol{X},\ast) \cong \Comm_{\Homeo(\bdry G)}(G).
 \]
\end{corollary}
\begin{proof}
    By \RefThm{qi-map-factors}, the image of $Q$ lies in $\Comm(G) \leq \QI(G)$. Because $X$ is homotopy equivalent to an aspherical CW complex, $Q$ defines an isomorphism $\HMod(\sol{X},\ast) \to \Comm(G)$ by \RefThm{ptd-homotopy-comm}. The result follows from \RefThm{hyp-boundary-comm}.
\end{proof}

\end{document}